\documentclass[12pt]{amsart}
\usepackage{color,graphicx,array, amssymb, amscd,slashed,comment}
\makeatletter
\@namedef{subjclassname@2020}{\textup{2020} Mathematics Subject Classification} 
\makeatother
\newtheorem{Theorem}{Theorem}[section]

\newtheorem{Corollary}[Theorem]{Corollary}
\theoremstyle{definition}
\newtheorem{Definition}{Definition}
\theoremstyle{remark}

\newtheorem{Remark}[Theorem]{Remark} 

\numberwithin{equation}{section}
\newcommand{\R}{\mathbb R}
\newcommand{\C}{\mathbb C}

\newcommand{\D}{\mathbb D}

\newcommand{\ISU}{{\rm SU}_{1, 1}}

\newcommand{\isu}{\mathfrak{su}_{1, 1}}

\newcommand{\ad}{\operatorname{Ad}}

\newcommand{\Nil}{{\rm Nil}_3}
\newcommand{\Mi}{\mathbb E^3_1}
\renewcommand{\Re}{\operatorname {Re}}
\renewcommand{\Im}{\operatorname {Im}}

\renewcommand{\l}{\lambda}
\newcommand{\be}{\begin{equation*}}
\newcommand{\ee}{\end{equation*}}

\begin{document}
\title{A duality for minimal surfaces in the  Heisenberg group}
 \author[S.-P.~Kobayashi]{Shimpei Kobayashi}
 \address{Department of Mathematics, Hokkaido University, 
 Sapporo, 060-0810, Japan}
 \email{shimpei@math.sci.hokudai.ac.jp}
 \thanks{The author is partially supported by JSPS KAKENHI Grant Number JP18K03265,  
 JP22K03304 and 
 Foreign Expert Key
 Support Program (Northeast Special) Grants-D20240219.}
\date{\today}
\pagestyle{plain}
\begin{abstract} We introduce and study the notion of a transformation surface associated with a 
nowhere-vertical minimal surface in the three-dimensional Heisenberg group, and prove its minimality and duality. Furthermore, by using the logarithmic derivative of the moving frame with respect to the spectral parameter, we derive the Sym formula for the dual minimal surface. 
\end{abstract}
\maketitle
\section*{Introduction}
A surface $f$ in the Euclidean 3-space $\R^3$ is said to be \textit{isothermic} if it admits conformal curvature line coordinates away from umbilic points. Such surfaces are classically characterized 
by the \textit{dual transformation}, also known as the \textit{Christoffel transformation}:
 namely, there exists another surface $f^*$ satisfying $df^* \wedge df = 0$ \cite[Definition 3.1]{KPP}.  
 Canonical examples of isothermic surfaces include 
 quadrics, surfaces of revolution, constant mean curvature surfaces, and 
 Bonnet surfaces.
 The  Christoffel dual surface $f^*$ is unique up to homothety, is itself isothermic, and satisfies the duality relation    $f^{**} = f$.

 In contrast, the notion of duality is far less understood in the setting of general Riemannian manifolds. It is therefore a natural and fundamental problem to identify geometrically significant classes of surfaces in specific Riemannian manifolds that admit a meaningful notion of duality.
 
 We address this problem in the context of the three-dimensional Heisenberg group
 $\Nil$, which carries a canonical left-invariant Riemannian metric and plays a central role in the geometry of three-dimensional Thurston geometries.
 We introduce and study the notion of a transformation surface associated with a nowhere-vertical minimal surface in $\Nil$, demonstrating its  minimality and  natural duality.
 
  The underlying idea is inspired by the correspondence between surfaces of nonzero constant mean curvature (CMC) in the Minkowski space $\R^{2,1}$ and  minimal surfaces in $\Nil$,
 established via their normal Gauss maps \cite{Figueroa}. 
 Since CMC surfaces in $\R^{2,1}$ admit duals, this correspondence
 naturally suggests a similar duality for minimal surfaces in $\Nil$. 
  
  Our construction of the dual surface in $\Nil$ begins with generating spinors associated with 
   a given nowhere-vertical minimal surface $f$ as developed in [6, 3] and reviewed in Section \ref{sc:Pr}.
 Then, by appropriately defining a pair of dual generating spinors, we construct another minimal surface
  $f^*$ in $\Nil$ and show that the duality relation $f^{**}=f$ holds (Theorem \ref{thm:dualmini}).
 In addition, following the approach in \cite{DIK:mini}, minimal surfaces in $\Nil$
 can be represented using the loop group method. We derive a representation formula for the dual minimal surface via the logarithmic derivative of the extended frame with respect to the spectral parameter, the so-called Sym formula (Theorem \ref{thm:Symdual}). We also present several explicit examples.

 This constitutes the first systematic construction of a duality for minimal surfaces in $\Nil$, developed via generating spinors and the loop group method.
 
 \section{Preliminaries}\label{sc:Pr}
 In this section, we briefly recall the necessary facts from \cite{DIK:mini}.
 We recall generating spinors of nowhere-vertical minimal surfaces in $\Nil$
 and a characterization of minimal surfaces by the harmonicity of the normal Gauss maps.

 The three-dimensional Heisenberg group $\Nil= (\R^3(x_1, x_2, x_3), \cdot)$ is 
 equipped  with the multiplication  $(x_1, x_2, x_3)\cdot (\tilde x_1, \tilde x_2, \tilde x_3):=
 (x_1+ \tilde x_1, x_2+ \tilde x_2, x_3 + \tilde x_3 + \frac12(x_1 \tilde x_2 - \tilde x_1 x_2))$
 and the left-invariant Riemannian metric $ds^2 = dx_1^2 + dx_2^2 + (dx_3 + \frac12 (x_2 dx_1 - x_1 dx_2))^2$. 
\subsection{Generating spinors for surfaces in $\Nil$}\label{sbsc:gspinors}
 Let us take an orthonormal basis 
 of the Lie algebra $\mathfrak{nil}_3$ of $\Nil$ by $\{e_1, e_2, e_3\}$. 
 Let $f: M \to \Nil$ be a conformal immersion from a Riemann surface $M$, then 
 the Maurer-Cartan form $f^{-1} d f$ can be 
 expanded as $f^{-1} d f = (f^{-1} f_z) dz + (f^{-1} f_{\bar z}) d \bar z$
 with $ f^{-1} f_z = \sum_{k =1}^3 \phi_k e_k$
 and $f^{-1} f_{\bar z} =\overline{f^{-1} f_z}  = \sum_{k =1}^3 \bar \phi_k e_k$.
 Here $(z= x + i y)$ are conformal coordinates, $\bar z= x - i y$ is its complex 
 conjugate, and the subscripts $z$ and $\bar z$ denote
 the partial differentiations with respect to $z$ and $\bar z$, respectively. 
 Since $f$ is a conformal immersion, $\phi_k (k =1, 2, 3)$ satisfies 
 \[
 \sum_{k=1}^3 \phi_k^2 =0\quad\mbox{and}\quad 
 \sum_{k =1}^{3}|\phi_k|^2 = \frac{1}{2}
 e^u\neq 0.
 \] 
 We note that the induced metric of $f$ is given by $ds^2 = e^u dz d \bar z$.
 Then using the {\it generating spinors} $\psi_1$ and $\psi_2$, the first equation 
 can be solved by
 \be
 \phi_1 = (\overline{\psi_2})^2 - \psi_1^2, \;\; 
 \phi_2 = i ((\overline{\psi_2})^2 + \psi_1^2), \;\; 
 \phi_3 = 2 \psi_1\overline{\psi_2}.
 \ee
 Let $N$ be the positively oriented unit normal vector field along $f$ and 
 denote  an unnormalized normal vector field $L$ by $L = e^{u/2} N$. We 
 define the support $h (dz)^{1/2} (d\bar z)^{1/2}$ by $h = \langle f^{-1} L, e_3\rangle$. 
 Then it is easy to compute $e^u$ and $h$ by the generating spinors $\psi_1$ and $\psi_2:$
 \begin{equation}\label{eq:uandh}
 e^u = 4 (|\psi_1|^2 + |\psi_2|^2)^2, \quad h = 2 (|\psi_1|^2-|\psi_2|^2).
 \end{equation}
 Moreover, let $e^{w/2}$ and $Q \,dz^2= 4 B\, dz^2$ be the Dirac potential and
 the Abresch-Rosenberg differential \cite{Abresch-Rosenberg}, given by 
\begin{equation}\label{eq:potandB}
 e^{w/2} = \mathcal U = \mathcal V = -\frac{H}{2}e^{u/2} + \frac{i}{4} h, 
 \quad   B = \frac{2 H + i}{4} \left(\langle f_{zz}, N \rangle + \frac{\phi_3^2}{2 H + i}\right).
\end{equation}
 It is known that the vector of  generating spinors $\tilde \psi = (\psi_1, \psi_2)$ 
 satisfies the so-called ``linear spinor system'', see \cite[Section 3]{DIK:mini} for details$:$
\begin{equation}\label{eq:linspinor}
 \tilde \psi_z = \tilde \psi \tilde U , \;\;
 \tilde \psi_{\bar z} =  \tilde \psi \tilde V, \;\;
\end{equation}
 where 
 \be 
 \tilde U =
 \begin{pmatrix}
 \frac{1}{2} w_z  + \frac{1}{2} H_{z} e^{-w/2+u/2}&
 - e^{w/2} \\ 
 B e^{-w/2} & 0
 \end{pmatrix}, \;\;
 \tilde V =
 \begin{pmatrix}
 0 & - \bar B e^{-w/2}\\
 e^{w/2} & \frac{1}{2}w_{\bar z}+\frac{1}{2} H_{\bar z}e^{- w/2 + u/2}
 \end{pmatrix}.
\ee
 We note that the second column of the first equation and 
 the first column of the second equation together 
 are the nonlinear Dirac equations, 
 that is,
 \begin{equation}\label{Dirac1}
\slashed{D} \begin{pmatrix} 
\psi_1
\\ 
\psi_2
\end{pmatrix} 
:=
\begin{pmatrix}
\partial_{z}\psi_{2}+\mathcal{U}\psi_1
\\
-\partial_{\bar z}\psi_1+\mathcal{V}\psi_2
\end{pmatrix} 
=
\left(
\begin{array}{c}
0
\\
0
\end{array}
\right), \quad  \mathcal U = \mathcal V = e^{w/2}.
\end{equation}

\subsection{Normal Gauss maps and minimal surfaces}
 We now identify the Lie algebra $\mathfrak{nil}_3$ with the 
 Euclidean $3$-space $\mathbb E^3$ via the natural basis $\{e_1, e_2,e_3\}$.
 Under this identification, 
 the map $f^{-1} N$ can be considered as a map into the 
 unit two-sphere $\mathbb S^2\subset \mathfrak{nil}_3$. 
 We now consider the \textit{normal Gauss map} $g$ of the surface $f$
 in $\Nil$, \cite{Figueroa}:
 The map $g$ is defined as the composition of the stereographic 
 projection $\pi$ from the south pole with
 $f^{-1} N$,  that is, $g = \pi \circ f^{-1} N: \D \to \C \cup \{\infty\}$
 and thus, applying the stereographic projection to  $f^{-1} N$, we obtain
 \begin{equation}\label{eq:Normal}
 g= \frac{\psi_2}{\overline{\psi_1}} \quad\mbox{and}\quad
 f^{-1}N=\frac{1}{1+|g|^2}
 \left(
 2\Re (g) e_1+2\Im (g) e_2+(1-|g|^2)e_3
 \right).
 \end{equation}
 The formula \eqref{eq:Normal} implies that 
 $f$ is nowhere-vertical if and only if $|g|<1$ or $|g|>1$.
 
 From now on we assume that the surface $f$ is nowhere-vertical and the unit normal  $f^{-1} N$ is upward, that is, 
 the $e_3$-component of $f^{-1} N$ is positive. Then the normal Gauss map 
 satisfies $|g|<1$, that is, it  takes values in the unit disk 
 $\mathbb D \subset \mathbb C$. Introducing the Poincar\'e metric on $\mathbb D$, 
 we consider $\mathbb D$ as the hyperbolic two-space $\mathbb H^2$ (the so-called Poincar\'e
 disk model).
\begin{Theorem}[Theorem 7 in \cite{Figueroa}]\label{thm:mincharact}
 Let $f : \D \to \Nil$ be a nowhere-vertical conformal immersion such that the normal 
 Gauss map $g$  satisfies $|g|<1$. Then the following statements are equivalent{\rm:}
 \begin{enumerate}
 \item $f$ is a minimal surface.
 \item The normal Gauss map $g$ for $f$ is a nowhere-holomorphic harmonic 
       map into the hyperbolic two-space $\mathbb H^2$.
 \end{enumerate}
 \end{Theorem}

\section{A duality for minimal surfaces}\label{sc:duality}
 In this section, we first define a pair of dual generating spinors for a pair of 
 generating spinors of a nowhere-vertical minimal surface in $\Nil$. Next, we 
 show that the pair of dual generating spinors is solvable and define a dual minimal surface in 
 $\Nil$.
\subsection{Dual generating spinors and dual minimal surfaces}
 Let $\psi_1$ and $\psi_2$ be the pair of generating spinors of 
 a nowhere-vertical conformal minimal surface $f$ in $\Nil$. Then we define the 
 \textit{pair of dual generating spinors} $\psi_1^*$ and $\psi_2^*$ as
\begin{equation}\label{eq:dualspinors}
 \psi_1^* = \frac{4\sqrt{- B}}{h}  \psi_2 ,\quad \psi_2^* = \frac{4 \sqrt{- \overline B}}{h} \psi_1,
\end{equation}
 where $B$ is the Abresch-Rosenberg differential in \eqref{eq:potandB}
 and $h$ is the support function in \eqref{eq:uandh}. 
 If $B \equiv 0$,  the corresponding surface is the so-called horizontal umbrellas, 
 then $\psi_1^* \equiv  \psi_2^* \equiv 0$; thus, we exclude such minimal surfaces from our consideration.  Note that since $B$ is holomorphic, the zeros of $B$ are isolated 
 if $B \not\equiv 0$. The following theorem constitutes the main result of this paper.
 \begin{Theorem}\label{thm:dualmini}
 Let $\psi_1$ and $\psi_2$ be the generating spinors of a nowhere-vertical conformal  
 minimal surface  $f$ in $\Nil$, and let $\psi_1^*$ and $\psi_2^*$ be the dual generating spinors 
 defined in \eqref{eq:dualspinors}. Further,  assume that the Abresch-Rosenberg differential 
 satisfies  $B \not\equiv 0$. Then the vector of generating spinors 
 $\tilde {\psi}^* = (\psi_1^*, \psi_2^*)$ 
 satisfies the following linear spinor system$:$
\begin{equation}\label{eq:linspinor*}
 \tilde \psi^*_z = \tilde \psi^* \tilde U^* , \;\;
 \tilde \psi^*_{\bar z} =  \tilde \psi^* \tilde V^*, \;\;
\end{equation}
 where 
 \be 
 \tilde U^* =
 \begin{pmatrix}
 \frac{1}{2} w^*_z &
 - e^{w^*/2} \\ 
 B^* e^{-w^*/2} & 0
 \end{pmatrix}, \;\;
 \tilde V^* =
 \begin{pmatrix}
 0 & - \overline{B^*} e^{-w^*/2}\\
 e^{w^*/2} & \frac{1}{2}w^*_{\bar z}
 \end{pmatrix},
\ee
 with 
\begin{equation}\label{eq:Diracpotdual}
 e^{w^*/2} = \frac{4 i|B|}{h}\quad\mbox{and}\quad B^* = B.
\end{equation}
 Therefore, there exists a conformal immersion $f^*$ in $\Nil$ such that the generating spinors of $f^*$
 are given by $\psi_1^*$ and $\psi_2^*$ with the Dirac potential $\mathcal U^* = \mathcal V^* =  e^{w^*/2}$ and the left translated unit normal to $f^*$ is equal to $f^{-1} N$ up to sign.
  Moreover, the following properties hold$:$
 \begin{enumerate}
 \item By choosing the left-translated unit normal $N^*$ so that $(f^*)^{-1} N^* = f^{-1} N \subset \mathfrak{nil}_3$, the metric $e^{u^*} |dz|^2$, the support $h^* |dz|$, the Abresch-Rosenberg 
 differential $B^* dz^2$ and  the normal Gauss map $g^*$ are  respectively given by 
\begin{equation}\label{eq:*uBh}
 e^{u^*} = \frac{4^4|B|^2}{h^4} e^u , \quad  h^*=  \frac{4^2 |B|}{h},\quad B^*= B
 \quad\mbox{and}\quad  g^* 
 =g,
\end{equation}
 where $e^{u} |dz|^2$, $h|dz|, B^* dz^2$ and  $g$ 
 are the metric, the support, the Abresch-Rosenberg  differential and  the normal Gauss map 
 of $f$, respectively.
 \item The surface $f^*$ is nowhere-vertical and minimal.

 \item It satisfies the duality, that is, $f^{**} = f$ holds up to a rigid motion.
 \end{enumerate}
 The surface $f^*$ will be called the {\rm dual minimal surface} to $f$.
\end{Theorem}
\begin{proof}
 By using the linear spinor system in \eqref{eq:linspinor} for $\tilde \psi = (\psi_1, \psi_2)$, we first 
 compute  the derivative of $\tilde \psi^*_1$ with respect to $z$:
 \begin{align*}
 \psi_{1, z}^* =  \left(\frac{4\sqrt{- B}}{h}\right)_z \psi_2+
 \frac{4\sqrt{- B}}{h}\psi_{2, z} = \left(\log \frac{4\sqrt{- B}}{h}\right)_z \psi_1^*  
 - \frac{i h B }{4 |B|} \psi_2^*.
 \end{align*}
  On the one hand, the definition of $e^{w^*/2}$ and $B^*$ in \eqref{eq:Diracpotdual} and the logarithmic derivative  of $e^{w^*/2}$ give
  \[
  \frac12 w^*_z = \left(\log \frac{4\sqrt{- B}}{h}\right)_z \quad\mbox{and} \quad 
  - \frac{i h B }{4 |B|}  = B^* e^{-w^*/2}.
  \]
  Thus the first column of $\tilde \psi^*_z$ in \eqref{eq:linspinor*} follows. Next, 
  we compute 
 \begin{align*}
 \psi_{2, z}^* &=  \left(\frac{4\sqrt{- \overline B}}{h}\right)_z \psi_1+ 
 \frac{4\sqrt{- \overline B}}{h}\psi_{1, z} \\ &=-(\log h)_z \frac{4 \sqrt{- \overline B}}{h}\psi_1  
 + \frac{4\sqrt{- \overline B}}{h}\left( \frac12 w_z \psi_1 + B e^{-w/2} \psi_2 \right). 
 \end{align*}
  Since $e^{w/2} = i h /4$, it can be simplified as
  \[
 \psi_{2, z}^* = \frac{4\sqrt{- \overline B}}{h} B e^{-w/2} \psi_2 = 
 e^{w^*/2} \psi_1^*. 
  \]  
 Thus  $\tilde \psi^*_z$ in 
 \eqref{eq:linspinor*} is obtained. A similar computation shows that 
 $\tilde \psi^*_{\bar z}$ of \eqref{eq:linspinor*} can be obtained. Therefore 
 $\tilde \psi^* = (\tilde \psi_1^*, \tilde \psi_2^*)$ defines a surface $f^*$ in $\Nil$.
  Note that the surface $f^*$ has the same left-translated tangent plane 
  to the original surface $f$, and thus the left translated unit normal to $f^*$ 
 is equal to $f^{-1} N$ up to sign.
 We now show the properties of $f^*$:
 
{\rm (1):} 
  The metric of $f^*$ can be computed by the definition of $\psi_1^*$ $\psi_2^*$, which is 
  given in \eqref{eq:*uBh} the first formula. 
 Since we choose $(f^*)^{-1} N^* = f^{-1} N$,
 the normal Gauss map $g^*$ is $g$.

  Recall that the support function $h^*$ is defined by  $h^* = \langle (f^*)^{-1} L^*, e_3\rangle$,
  it follows that 
  \[
h^* =e^{u^*/2} \langle f^{-1} N, e_3\rangle   = \frac{4^2 |B|}{h}.
  \]
 Moreover, the Abresch-Rosenberg differential can be computed easily as in \eqref{eq:*uBh}.
 
{\rm (2):} The minimality $H^*=0$ follows from the form of the Dirac potential 
\[
\mathcal U^* = \mathcal V^*  = e^{w^*/2 } = -\frac{H^*}2 e^{u^*/2} + \frac{i}4 h^* = 
\frac{4 i |B|}{h} \in i \mathbb R,
\]
 and it is clear that $f^*$ is nowhere-vertical by the construction.
 
{\rm (3):} The surface $f^{**} = (f^*)^*$ is given by the pair of double dual generating spinors 
\begin{align}\label{eq:ddualspinors}
 \psi_1^{**} = \frac{4 \sqrt{- B^*}}{h^*} \psi_2^*, \quad \psi_2^{**}  = \frac{4 \sqrt{- \overline {B^*}}}{h^*} \psi_1^*.
\end{align}
 Then the relations in \eqref{eq:dualspinors} and \eqref{eq:*uBh} show 
 $\psi_1^{**} = \psi_1$ and $\psi_2^{**} = \psi_2$ hold, and therefore, 
 the pair of generating spinors $(\psi_1^{**},  \psi_2^{**})$ defines 
 the original minimal surface $f$ 
 up to the rigid motion. This completes the proof.
\end{proof}

\begin{Remark}
 From \eqref{eq:dualspinors}, it is evident that zeros of $B$ are the singularities of the dual minimal surface $f^*$. By using the dual generating spinors $\psi_1^*$ and $\psi_2^*$, the metric $e^{u^*/2}$,
 the support $h^*$ and the normal Gauss map $g^*$ can be represented by 
  \begin{equation}\label{eq:locaexpression}
  e^{u^*} = 4 (|\psi_1^*|^2 + |\psi_2^*|^2)^2, \quad 
  h^* = 2 (|\psi_2^*|^2 - |\psi_1^*|^2), \quad 
  g^* = - \frac{\psi_1^* }{\overline{\psi_2^*}}.
 \end{equation}
 Here we choose the branch of $\sqrt{-B}$ so that $\overline{\sqrt{-\overline{B}}}=-\sqrt{-B}$.
\end{Remark}
 An original minimal surface $f$ and the dual minimal  surface $f^{*}$ in 
 Theorem \ref{thm:dualmini} can be summarized in Table \ref{tb:geo}.
\begin{table}[tp]
\extrarowheight=1mm
\begin{tabular}{c|c|c|c|c}
 {\small surface} &{\small mean curvature} & 
 {\small  metric} & {\small holo. AR-differential} 
 &{\small support} \\[1mm]\hline
 $f$ & $H =0$ & $e^u|dz|^2$ & $B dz^2$ & 
 $ h |dz|$ \\[1mm]\hline
  $f^*$ &$H =0$ &$\frac{4^4 |B|^2 }{h^4}  e^u|dz|^2$ & $B dz^2$ &$\frac{4^2 |B|}{h} |dz|$\\[1mm]
\end{tabular}
\vspace{0.3cm}
\caption{
 The geometric data of a minimal surface $f$ and its dual $f^*$. 
 }\label{tb:geo}
\end{table}
\section{The dual Sym formula}
 In this section, we show that the pair of minimal surfaces in 
 Theorem \ref{thm:dualmini} will be obtained by the so-called \textit{Sym-formula}. 
 On the level of the Sym-formula, the concept of duality for minimal surfaces admits 
 a simple algebraic interpretation.
\subsection{Flat connections and extended frames}
 We recall the family of Maurer-Cartan forms $\alpha^{\lambda}$ is 
 defined by 
 \begin{equation}\label{eq:alpha}
 \alpha^{\l} = \tilde U^{\l} dz + \tilde V^{\l} d\bar z, \quad \lambda \in \mathbb S^1,
 \end{equation}
 with 
 \begin{equation}\label{eq:U-V1lambda}
 \tilde U^{\lambda} =
 \begin{pmatrix}
 \frac{1}{4} w_z  + \frac{1}{2} H_{z} e^{-w/2+u/2}&
 - \l^{-1}e^{w/2} \\ 
 \l^{-1}B e^{-w/2} &  -\frac{1}{4} w_z
 \end{pmatrix}, \;\;
 \tilde V^{\lambda} =
 \begin{pmatrix}
  -\frac{1}{4} w_{\bar z} & - \l \bar B e^{-w/2}\\
 \l e^{w/2} & \frac{1}{4}w_{\bar z}+\frac{1}{2} H_{\bar z}e^{- w/2 + u/2}
 \end{pmatrix}.
 \end{equation}
 Note that $\tilde U^{\lambda}|_{\lambda =1} = \tilde U$ and 
 $\tilde V^{\lambda}|_{\lambda =1} = \tilde V$ in \eqref{eq:linspinor}.
 In addition to a characterization of minimal surfaces in terms of the harmonic
 normal Gauss map as in Theorem \ref{thm:mincharact}, the following theorem 
 has been known.
 \begin{Theorem}[Theorem 5.3 in \cite{DIK:mini}]\label{thm:mincharact2}
 Let $f : \D \to \Nil$ be a nowhere-vertical conformal immersion and $\alpha^{\l}$
 the $1$-form defined in \eqref{eq:U-V1lambda}.
 Moreover, assume that the unit normal $f^{-1} N$ is upward.
 Then the following statements are equivalent{\rm:}
 \begin{enumerate}
 \item $f$ is a minimal surface.
 \item $d + \alpha^{\l}$ is a family of flat connections on $\D \times  \ISU$.
 \end{enumerate}
 \end{Theorem}
 From Theorem \ref{thm:mincharact2}, the family of Maurer-Cartan form 
 $\alpha^{\lambda}$ in \eqref{eq:alpha}
 can be simplified as follows 
 \begin{equation}\label{eq:alpha2}
 \alpha^{\l} = U^{\l} dz + V^{\l}d \bar z, 
 \end{equation}
 with
 \begin{equation*}
U^{\l}=
 \begin{pmatrix}
 \frac{1}{4} w_z &
 - \l^{-1}e^{w/2} \\ 
 \l^{-1}B e^{-w/2} &  -\frac{1}{4} w_z
 \end{pmatrix}, \quad 
V^{\l}=
 \begin{pmatrix}
  -\frac{1}{4} w_{\bar z} & - \l \bar B e^{-w/2}\\
 \l e^{w/2} & \frac{1}{4}w_{\bar z}
 \end{pmatrix},
 \end{equation*}
 and we give the following definition.
\begin{Definition}[Definition 1 in \cite{DIK:mini}]
 Let $f$ be a nowhere-vertical minimal surface in $\Nil$ and 
 $F$ as above 
 the corresponding $\ISU$-valued solution to the equation
 $F^{-1} d F = \alpha^{\lambda}, (\lambda \in \mathbb S^1)$, 
 where $\alpha^{\lambda}$ is defined by \eqref{eq:alpha2}.
  Then $F$ is called 
 \textit{extended frame} of the minimal surface $f$.
\end{Definition}
 In particular, we can express the extended frame 
 associated with respect to the generating spinors 
 $\psi_1$ and $\psi_2$ for a minimal surface;
 \begin{equation}\label{eq:extframin}
 F(\lambda) =\frac{1}{\sqrt{|\psi_1(\lambda)|^2-|\psi_2(\lambda)|^2}} 
 \begin{pmatrix}
 \sqrt{i}^{-1} \psi_1(\lambda) & 
 \sqrt{i}^{-1} \psi_2(\lambda) \\ 
 \sqrt{i} \;\overline{\psi_2(\lambda)} & 
  \sqrt{i}\; \overline{ \psi_1(\lambda)}
 \end{pmatrix}. 
 \end{equation}

 Here $\psi_1(\lambda)$ and $\psi_2(\lambda)$ denote families of functions such that 
 $\psi_1(\lambda=1) = \psi_1$ and $\psi_2(\lambda=1 )= \psi_2$, which 
 are the generating spinors of the minimal surface $f$, respectively.
\begin{Remark}
 In fact the extended frame $F$ of a nowhere-vertical minimal surface is the 
 extended frame of a non-conformal harmonic map into the hyperbolic two-space
 $\mathbb H^2$, and the loop group method for such harmonic maps has been 
 studied in \cite{BRS:Min}.
\end{Remark}

\subsection{The dual Sym formula}
 We first identify the Lie algebra 
 $\mathfrak{nil}_3$ with the 
 Lie algebra $\isu$ as a \textit{real vector space}. 
 In $\isu$, we choose the following basis:
\begin{equation}\label{eq:basis}
 \mathcal{E}_1 = \frac{1}{2} \begin{pmatrix} 0 & i \\ -i &0 \end{pmatrix}, \;\;
 \mathcal{E}_2 = \frac{1}{2} \begin{pmatrix} 0 & -1 \\ -1 & 0 \end{pmatrix}\;\;
 \mbox{and}\;\;\;
 \mathcal{E}_3 = \frac{1}{2} \begin{pmatrix} -i & 0\\ 0 &i \end{pmatrix}.
\end{equation}
 One can see that $\{\mathcal{E}_1,\mathcal{E}_2,\mathcal{E}_3\}$ is an 
 orthogonal basis 
 of $\isu$ with timelike vector 
 $\mathcal{E}_3$. A linear isomorphism $\Xi:\mathfrak{su}_{1,1}\to 
 \mathfrak{nil}_3$ is then given by
 \begin{equation}\label{eq:Nilidenti}
\mathfrak{su}_{1,1} \ni 
x_1 \mathcal{E}_1 + x_2 \mathcal{E}_2 + x_3 \mathcal{E}_3
\longmapsto
x_1 e_1 +  x_2 e_2 +  x_3 e_3 \in \mathfrak{nil}_{3}.
\end{equation}
 Note that the linear isomorphism $\Xi$ is not a Lie algebra 
 isomorphism. 
 Next, we realize $\Nil$ in $\mathrm{GL}_4 \R$ by 
 $\iota (x_1, x_2, x_3)= 
 e^{x_1} E_{11} + \sum_{i=2}^4 E_{ii} + x_1 E_{23} + (x_3 + \frac12 x_1 x_2) E_{24} + x_2 E_{34}
 $, where $E_{ij} (i, j = 1, 2, 3, 4)$ denote $4$ by $4$ matrices with the $ij$-entry $1$ and  
 all others $0$. Consider the exponential map 
 $\exp:\mathfrak{nil}_3\to \mathrm{Nil}_3$ is given by 
\[
\exp (x_1 e_1 + x_2 e_2 + x_3 e_3) = e^{x_1} E_{11} + \sum_{i=2}^4 E_{ii} + x_1 E_{23} + (x_3 + \frac12 x_1 x_2) E_{24} + x_2 E_{34}.
\]
 We define a smooth bijection  
 $\Xi_{\rm nil}:\isu \to \Nil$ by $\Xi_{\rm nil}:=\exp \circ \Xi$.
 By using the logarithmic derivatives
 for the extended frame with respect to $\lambda$, denoted by $\partial_{\lambda}$,
 we now prove the Sym-formula.

\begin{Theorem}\label{thm:Symdual}
 Let $F$ be the extended frame
 for some nowhere-vertical minimal surface $f$ with $B\not \equiv 0$ in $\Nil$,  and define
 $m_{\pm}$ and $N_m$ respectively the maps 
 \begin{equation}\label{eq:SymMin}
 m_{\pm}=-i \lambda (\partial_{\lambda} F) F^{-1} 
 \pm N_m\;\;
 \mbox{and} \;\;
 N_{m}= \frac{i}{2} \ad (F) \sigma_3.
 \end{equation}
 Moreover, define a map  $f_{\pm}^{\lambda}:\mathbb{D}\to \mathrm{Nil}_3$ by
 $f_{\pm}^{\lambda}:=\Xi_{\mathrm{nil}}\circ \hat{f_{\pm}^{\lambda}}$
 with
\begin{equation}\label{eq:symNil}
 \hat f_{\pm}^{\lambda} = 
    \left.
    \left(m_{\pm}^o -\frac{i}{2} \lambda (\partial_{\lambda} m_{\pm})^d\right)
    \;\right|_{\lambda \in \mathbb{S}^1}, 
\end{equation}
 where the superscripts ``$o$'' and ``$d$'' denote the off-diagonal and 
 diagonal part, 
 respectively. Then, for each $\lambda \in \mathbb{S}^1$, 
 the maps $f_{\pm}^{\lambda}$ are minimal surfaces in $\Nil$ and 
 $N_m$ is the normal Gauss map of $f_{\pm}^{\lambda}$. 
 In particular, $f_{-}^{\lambda}|_{\lambda =1}$ is the original minimal surface $f$ 
 up to a rigid motion and $f_{+}^{\lambda}|_{\lambda =1}$ is the 
 dual minimal surface $f^*$. Moreover, for each $\lambda \in \mathbb S^1$, $f_{-}^{\lambda}$ and 
 $f_{+}^{\lambda}$ are dual to each other.
\end{Theorem}
\begin{proof}
 The statement for $f_-^{\lambda}|_{\lambda=1}$ has been proved in Theorem 6.1 in \cite{DIK:mini}, that is
 $f_-^{\lambda}|_{\lambda=1}$ is the original minimal surface $f$ up to rigid motion.
 We now consider $f_+^{\lambda}|_{\lambda =1}$. A straightforward computation shows that 
\begin{eqnarray}\label{eq:derivative}
 \partial_z m_+ &=& \ad (F)
 \left(-i \lambda \partial_{\lambda} U^{\lambda}+\frac{i}{2}[U^{\lambda}, \sigma_3] \right) \\ \nonumber &=&
 2 i\lambda^{-1} B e^{-w/2} \ad (F)
 \begin{pmatrix} 
 0 & 0 \\ 1 & 0
 \end{pmatrix}  \\ 
\nonumber
&=& 
\nonumber
- \frac{16 \lambda^{-1} B}{h^2} 
\begin{pmatrix}
 - \psi_2 (\lambda) \overline{\psi_1 (\lambda)} & - i \psi_2(\lambda)^2 \\
 -i \overline{\psi_1 (\lambda)}^2 & \psi_2 (\lambda) \overline{\psi_1 (\lambda)}
\end{pmatrix}.
\end{eqnarray}
 Note that $e^{\omega/2} = \frac{i}4 h \in i \R$
 is the Dirac potential for $f$.
 Since the dual generating spinors $\psi_1^*$ and $\psi_2^*$ are defined by 
\begin{equation}\label{eq:psi+}
 \psi_1^* = \frac{4 \sqrt{- B}}{h}  \psi_2 ,\quad \psi_2^* = \frac{4 \sqrt{- \overline B}}{h} \psi_1,
\end{equation}
 thus 
\begin{equation}\label{eq:phiequation}
\partial_z m_+= 
\phi_{1}^+(\lambda) \mathcal{E}_1 + \phi_{2}^+(\lambda) \mathcal{E}_2
-i \phi_3^+(\lambda) \mathcal E_3
\end{equation}
 with
\begin{equation*}
 \phi_{1}^+(\lambda) = \lambda^{-1} (\overline{\psi_2^*(\lambda)}^2 - \psi_1^*(\lambda)^2),\;\;
 \phi_{2}^+(\lambda) = i\lambda^{-1} \left(\overline{\psi_2^*(\lambda)}^2 + \psi_1^*(\lambda)^2\right)
\end{equation*}
and
\begin{equation*}
\phi_3^+(\lambda) = 2 \lambda^{-1} \psi_1^*(\lambda) \overline{\psi_2^*(\lambda)}.
 \end{equation*}
 Thus using \eqref{eq:derivative}, 
 the derivative of $m_+$ with respect to $z$ and $\lambda$ can be 
 computed as
\begin{eqnarray}\label{eq:derivative2}
 \partial_z ( i \lambda (\partial_{\lambda} m_+)) = 
 i \lambda \partial_{\lambda} (\partial_z m_+) & = &
 i \lambda \partial_{\lambda} \left(2 i \lambda^{-1} B e^{-w/2} \ad (F) 
 \begin{pmatrix} 
 0 & 0 \\ 1 & 0
 \end{pmatrix}
 \right), \\
 &=& -i(\partial_z m_+)
   -\left[m_+ - N_m, \partial_z m_+\right].
\nonumber
\end{eqnarray}
 Here $[a, b]$ denotes the 
 usual bracket of matrices, that is, $[a, b] = ab -b a$.
 Using \eqref{eq:derivative}, we have 
\begin{equation*}
 \left[N_m, \partial_z m_+ \right]^d 
 = - i (\partial_z m_+)^d 
\end{equation*}
 and
\begin{equation*}
 - [m_+, \partial_z m_+]^d = \left(\phi_1^+(\lambda) 
 \int \phi_2^+(\lambda) \, dz - \phi_2^+(\lambda) \int \phi_1^+(\lambda) \, dz \right) 
 \mathcal{E}_3.
\end{equation*}
 Thus we have
\begin{equation}\label{eq:partialfm}
\partial_z \left( -\frac{i \lambda \partial_{\lambda} m_+}{2}^d\right) 
 = \left(
  \phi_3^+(\lambda) -\frac{1}{2} \phi_1^+(\lambda) \int \phi_2^+(\lambda)\, dz  
  + \frac{1}{2}\phi_2^+(\lambda) \int \phi_1^+(\lambda)\, dz 
  \right) 
 \mathcal{E}_3.
\end{equation}
 Therefore, combining \eqref{eq:phiequation} and \eqref{eq:partialfm}, we obtain
$$
  \partial_z \hat f_+^{\lambda} = \phi_1^+(\lambda) \mathcal{E}_1 + \phi_2^+(\lambda) \mathcal{E}_2 
  +\left(\phi_3^+(\lambda)  - \frac{1}{2}\phi_1^+(\lambda) \int \phi_2^+(\lambda)\, dz
  + \frac{1}{2}\phi_2^+(\lambda) \int \phi_1^+(\lambda)\, dz
 \right) \mathcal{E}_3.
$$
 We now use the identification \eqref{eq:Nilidenti} with the left translation 
 $(f_+^{\lambda})^{-1}$, that is, 
 \begin{equation}\label{eq:immersion}
 (f_+^{\lambda})^{-1} \partial_z f_+^{\lambda} =  \phi_1^+(\lambda) e_1 + 
 \phi_2^+(\lambda) e_2 + \phi_3^+(\lambda) e_3.
\end{equation}
 Thus $\psi_1^*(\lambda)$ and $\psi_2^*(\lambda)$ 
 are spinors for $f_+^{\lambda}$ for 
 each $\lambda \in \mathbb{S}^1$. 
 In particular, the function
\[
 \frac{i}{2}(|\psi_1^*(\lambda)|^2 
 - |\psi_2^*(\lambda)|^2)
 = e^{w^*/2} = \frac{4 i|B|}{h}
 \]
 does not 
 depend on $\lambda$ and implies that 
 the mean curvature $H$ is equal to zero. 
 Moreover, the conformal factor of the induced 
 metric of $f_+^{\lambda}$ is given by  
\[
 e^{u^*} = 4(|\psi_1^*(\lambda)|^2 + |\psi_2^*(\lambda)|^2)^2 = \frac{4^4 |B|^2}{h^4} e^{u}.
\]
 This metric is degenerate at points where $B$ vanishes.
 Note that because of the holomorphicity of $B$, zeros of $B$ are isolated.
 Thus the map $f_+^{\lambda}$ actually defines a minimal surface 
 in $\Nil$ for each $\lambda \in \mathbb{S}^1$ where the points $B \neq 0$. 
 This completes the proof.
\end{proof}
\subsection{Examples}\label{subsc:examples}
 The general construction of minimal surfaces in $\Nil$ using holomorphic potentials and the loop group decomposition has been described in \cite[Sections 7 and 8]{DIK:mini}. We now present several explicit examples of minimal surfaces in 
 $\Nil$ and their duals, constructed via holomorphic potentials as introduced in \cite[Section 9]{DIK:mini}. 
 
\subsection*{Hyperbolic paraboloid and its dual} 
 The following holomorphic potential generates the hyperbolic paraboloid $x_3 = \frac{1}{2}x_1x_2$, as described in
 \cite[Section 9.2]{DIK:mini}:
\[
\xi_- = -\frac{i}{4} \lambda^{-1} \begin{pmatrix} 0 & 1 \\ 1 & 0 \end{pmatrix} dz.
\]

The associated family of minimal surfaces is given by
\[
f^\lambda = \left( -2i(p - p^*),\ -\sinh(2(p + p^*)),\ i(p - p^*)\sinh(2(p + p^*)) \right),
\]
where $p = -\frac{i}{4} \lambda^{-1} z$ and $p^* = \frac{i}{4} \lambda \bar{z}$.

The extended frame of this surface is
\[
F = \begin{pmatrix}
\sqrt{i}^{-1} \cosh(p + p^*) & \sqrt{i}^{-1} \sinh(p + p^*) \\
\sqrt{i} \sinh(p + p^*) & \sqrt{i} \cosh(p + p^*)
\end{pmatrix}.
\]

Using the Sym formula \ref{thm:Symdual}, the dual minimal surface is
\[
f^\lambda_+ = \left( -2i(p - p^*),\ \sinh(2(p + p^*)),\ -i(p - p^*)\sinh(2(p + p^*)) \right).
\]
 It is straightforward to verify that $f^\lambda_+$ again represents the same hyperbolic paraboloid, 
 showing that the surface is \emph{self-dual}.

\subsection*{Helicoid and its dual}
The holomorphic potential
\[
\eta = D\,dz,\quad D = \begin{pmatrix} c & a\lambda^{-1} + b\lambda \\ -a\lambda - b\lambda^{-1} & -c \end{pmatrix},\quad a = -b,\quad c = \frac{1}{2},
\]
generates a helicoid when the initial condition is the identity matrix \cite[Section 9.3]{DIK:mini}. Applying the Sym formula to construct the dual surface 
again yields the same helicoid, confirming that it is also \emph{self-dual}, see 
Figure \ref{fig:0}.

  \begin{figure}[t]
  \begin{center}
 \begin{tabular}{c}
      \begin{minipage}{0.45\hsize}
        \begin{center}
   \includegraphics[width=0.6\textwidth]{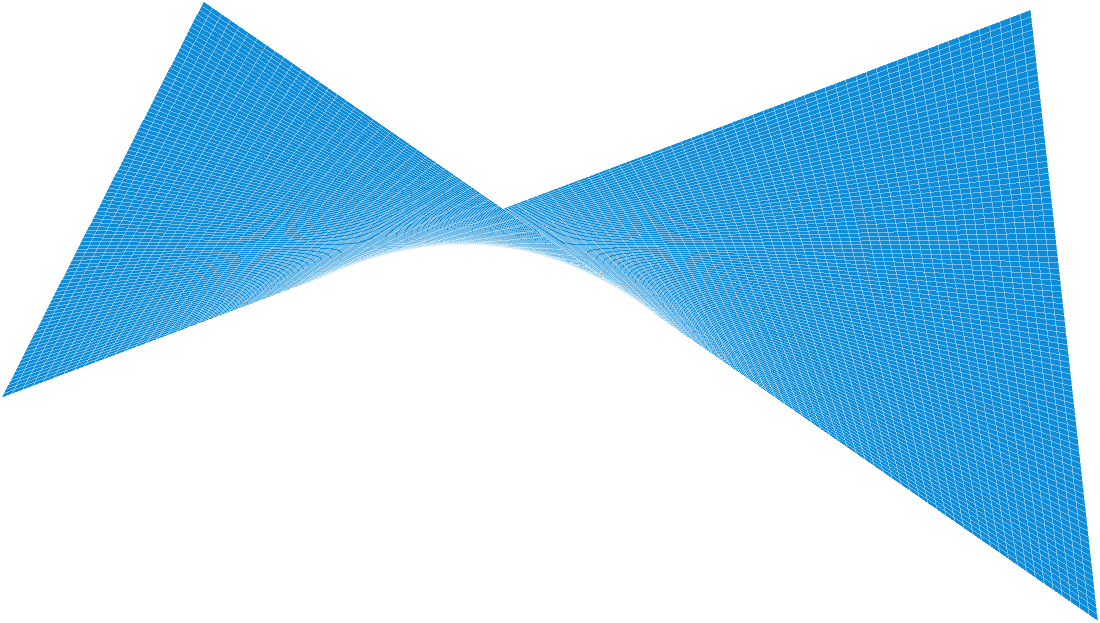}
        \end{center}
      \end{minipage}

      \begin{minipage}{0.55\hsize}
        \begin{center}
   \includegraphics[width=0.5\textwidth]{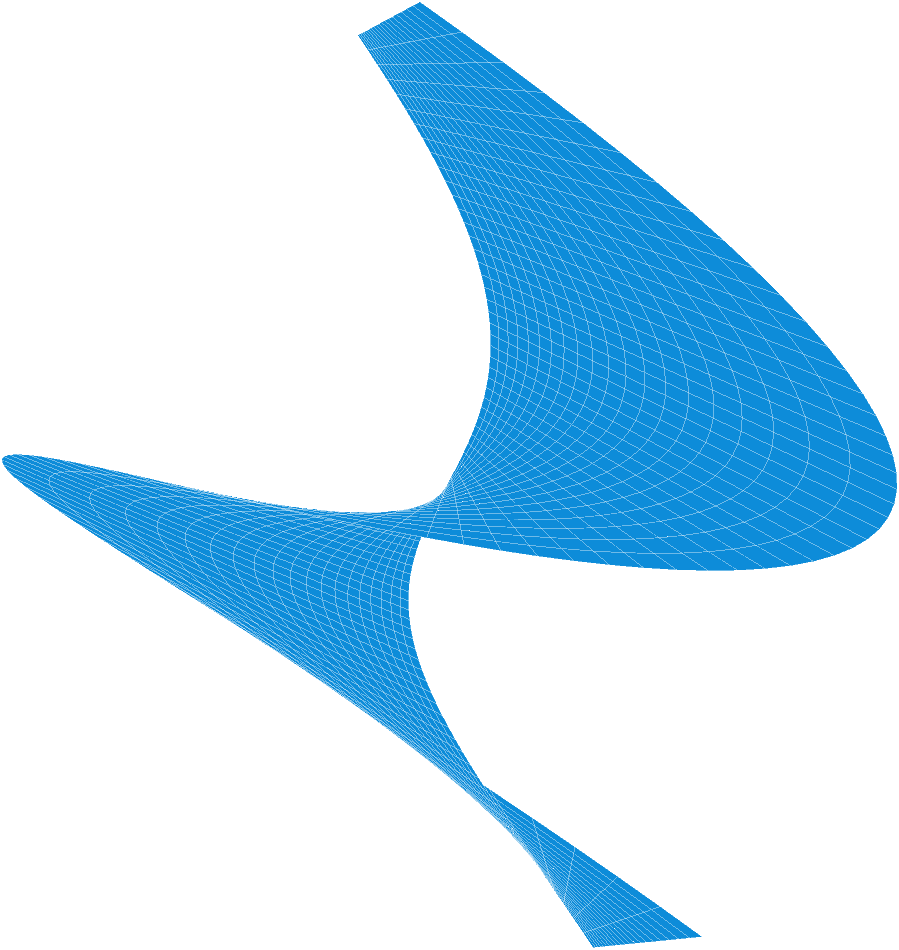}
        \end{center} 
      \end{minipage}
\end{tabular}
\caption{A hyperbolic paraboloid (left) and a helicoid (right). The both surfaces are self-dual.
}
    \label{fig:0}
  \end{center}
\end{figure}
 \subsection*{Smyth-type surface and its dual}
Consider the holomorphic potential
\[
\xi_- = \lambda^{-1} \begin{pmatrix} 0 & 1 \\ z^k & 0 \end{pmatrix} dz,\quad (k \in \mathbb{Z}_+),
\]
which generates a class of minimal surfaces referred to as \emph{Smyth-type surfaces} in $\mathrm{Nil}_3$. The Abresch-Rosenberg differential in this case vanishes at $z = 0$, and as a result, the corresponding dual surface $f_+^\lambda$ exhibits a \emph{branch point} at the origin, see Figure \ref{fig:1}.

 \begin{figure}[tbp]
\begin{tabular}{c}
      \begin{minipage}{0.5\hsize}
        \begin{center}
 \includegraphics[width=0.7\textwidth]{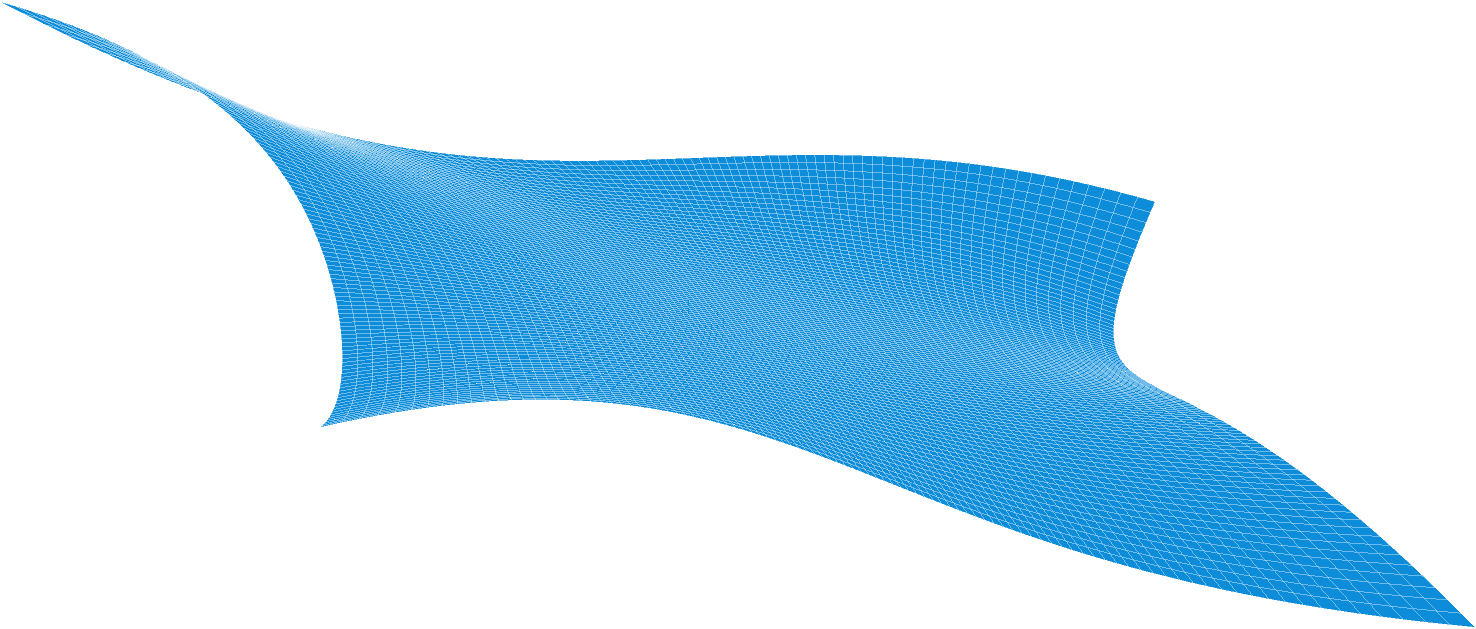}
        \end{center}
      \end{minipage}

      \begin{minipage}{0.5\hsize}
        \begin{center}
 \includegraphics[width=0.5\textwidth]{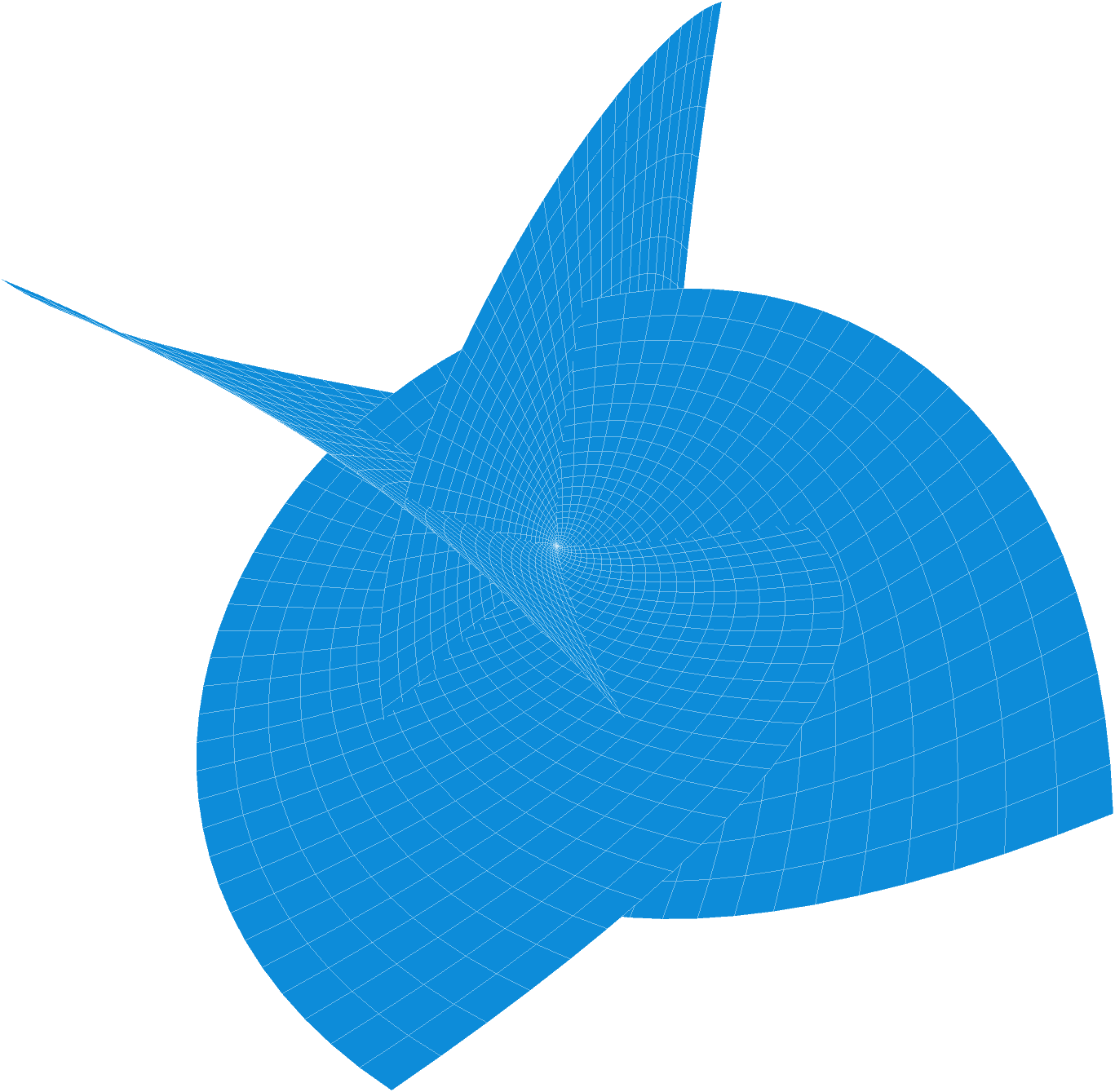}        
        \end{center} 
      \end{minipage}
\end{tabular}
\caption{A Smyth type surface (left) and its dual (right). A branch point at the origin on the right picture 
 appears due to the vanishing of the Abresch-Rosenberg differential.
 Figures generated using software by Brander \cite{Br:Matlab}.}
 \label{fig:1}
\end{figure}

\bibliographystyle{plain}
\def\cprime{$'$}

\end{document}